\documentclass[twoside,11pt]{article}                                           
\usepackage{ctexheading}\ctexset{section/format+=\centering}
\usepackage{indentfirst}
\usepackage{titlesec}                                                  
 \titleformat{\section}{\centering\small}{\thesection\,.}{0em}{}           
\usepackage{amssymb}
\usepackage{amsmath}
\usepackage{cite}
\usepackage{amsthm}
\usepackage{tikz-cd}
\usepackage{amsfonts}
\usepackage{enumitem}
\usepackage{geometry}                                                      
\usepackage{graphicx}                                                      
\usepackage{titlesec} \usepackage{titletoc}                                
\usepackage{hyperref} \hypersetup{colorlinks=true, linkcolor=blue, citecolor=green,} 
\usepackage{abstract}
\usepackage{fancyhdr}

\fancypagestyle{fancy}{%
    \fancyhf{} 
    \fancyhead[LE,RO]{\footnotesize\thepage} 
    \fancyhead[CE]{\footnotesize CHENG ZHANG} 
    \fancyhead[CO]{\footnotesize A Numerical Generalized Canonical Bundle Formula}

}
\title{\textbf{\normalsize\vspace{-1.5mm}\MakeUppercase{{A Numerical Generalized Canonical Bundle Formula}}}}
\author{\small{CHENG ZHANG}}
\date{}

\newtheorem{definition}{Definition}[section]                          
\newtheorem{theorem}[definition]{Theorem}
\newtheorem{lemma}[definition]{Lemma}
\newtheorem{proposition}[definition]{Proposition}

\newtheorem{remark}[definition]{Remark}

\titlecontents{section}[1cm]{\normalsize}{\contentslabel{2em}}{\hspace{-9mm}}{\titlerule*[0.5pc]{}\contentspage\hspace*{1cm}}

\begin{document}
\pagestyle{fancy}
\maketitle

\vspace{-20mm}

\begin{abstract}
\noindent\textsc{Abstract}. We establish the generalized canonical bundle formula in the numerical setting.
\end{abstract}


\rmfamily                                                             
\tableofcontents

\textsc{\section{Introduction}}

We work in the framework of generalized pairs introduced by Birkar--Zhang~\cite{BZ16}.
The canonical bundle formula plays a central role in birational geometry, particularly in inductive arguments on the dimension of varieties.
Established and developed by Kawamata~\cite{Kaw98}, Ambro~\cite{Amb04, Amb05}, and Fujino--Gongyo~\cite{FG14}, this formula was extended to the setting of generalized pairs by Filipazzi~\cite{Fil20},
and recently to the complex analytic setting~\cite{HP24, Has26}.
\par Meanwhile, numerical geometry plays an essential role in minimal model theory, especially in the setting of generalized pairs.
Indeed, central conjectures for usual pairs, such as the nonvanishing and abundance conjectures,
have analogous counterparts in the generalized pair framework only after replacing linear equivalence by numerical equivalence (see~\cite{BH14, HL18, LP20} for related discussions).
This numerical aspect of generalized pairs motivates a re-examination of classical theorems that were originally established under linear conditions.
\par In this paper, we establish a numerical analogue of Filipazzi's theorem for generalized pairs.

\vspace{2mm}

\begin{theorem}\label{n-g-cbf}
     Let $(X,B,\mathbf{M})$ be a projective g-sub-pair and $f:X\rightarrow Z$ a contraction to a normal projective variety $Z$. Suppose that\vspace{-2mm}
     \begin{itemize}
          \item $Z$ is $\mathbb{Q}$-factorial,\vspace{-3mm}
          \item $(X,B,\mathbf{M})$ is g-lc over the generic point of $Z$, and\vspace{-3mm}
          \item $K_X+B+\mathbf{M}_X\equiv_Z0$.\vspace{-2mm}
     \end{itemize}
     Then, there exists a g-sub-pair $(Z,B_Z,\mathbf{N})$ such that
     $$K_X+B+\mathbf{M}_X\equiv f^*(K_Z+B_Z+\mathbf{N}_Z).$$
     Moreover, if $(X,B,\mathbf{M})$ is a g-pair (resp. g-lc, g-klt, g-sub-lc, g-sub-klt), then so is $(Z,B_Z,\mathbf{N})$.
\end{theorem}

In the case where $\mathbf{M}=0$, the theorem holds straightforwardly without the $\mathbb{Q}$-factoriality assumption (see Theorem~\ref{cbf} for a detailed proof),
as numerical triviality coincides with linear triviality after reduction to the lc case \cite{Gon11}. 
The situation becomes much more complicated when $\mathbf{M} \neq 0$. The first challenge one encounters is the descent of divisors.
Specifically, to construct the b-divisors $B_Z$ and $\mathbf{N}$, one must determine whether a divisor $K_X+B+\mathbf{M}_X$
that is relatively numerically trivial over $Z$ is numerically equivalent to a pullback from $Z$.
While this property is automatic for $\mathbb{R}$-linear equivalence, it is much more subtle for numerical equivalence. 
To address this problem, Lehmann \cite{Leh15} provided a satisfactory result for pseudo-effective $\mathbb{R}$-Cartier divisors.
However, we cannot directly apply it here, as the generalized canonical divisor may fail to be pseudo-effective. 
Instead, we incorporate Lehmann's result into Filipazzi's strategy.
Specifically, we divide the problem into two cases. If $\mathbf{M}$ is numerically trivial over the generic fiber, we use weak semistable reduction to reduce to the case
where $\mathbf{M}$ is numerically trivial over the entire base $Z$, and then apply Lehmann's theorem to $\mathbf{M}$ while simultaneously running a relative MMP to obtain a model
where $K_X+B$ is also numerically trivial over the base.
On the other hand, if $\mathbf{M}$ is not numerically trivial over the generic fiber, we proceed by induction on the relative dimension of $X\to Z$ using the structure of the Mori fiber space.
We also apply the same strategy to establish the b-nefness of $\mathbf{N}$.

\textsc{\section{Preliminaries}\label{Sec: Preliminaries}}

In this section we collect definitions, and show some important results. We will follow the standard notations and terminologies in \cite{KM98, BCHM10} and use them freely.
\vspace{2mm}

\noindent\textbf{Numerical dimension.}
Let $X$ be a normal projective variety. Let $D$ be an $\mathbb{R}$-Cartier divisor and $A$ a Cartier divisor on $X$.
Set
$$\sigma (D;A)=\text{max}\{k\in \mathbb{Z}_{\ge 0}\ |\ \limsup_{m \to \infty}\frac{h^0(X,\mathcal{O}_X(\llcorner mD\lrcorner +A))}{m^k}>0 \}$$
if $h^0(X,\mathcal{O}_X(\llcorner mD\lrcorner +A))>0$ for infinitely many $m\in\mathbb{N}^*$, and set
$$\kappa_\sigma (X,D)=\text{max}\{\sigma (D;A)\ |\ \text{A is a Cartier divisor on}\ X \}.$$
$\kappa_\sigma (X,D)$ is called the \emph{numerical dimension} of $D$.
\par Let $X\rightarrow Z$ be a projective morphism from a normal variety to a variety, and $D$ an $\mathbb{R}$-Cartier divisor on $X$.
Then the \emph{relative numerical dimension} of $D$, denoted by $\kappa_\sigma(X/Z,D)$, is defined by $\kappa_\sigma(F,D|_F)$,
where $F$ is a sufficiently general fiber of the Stein factorization of $X\rightarrow Z$.
Note that $\kappa_\sigma(F,D|_F)$ does not depend on the choice of $F$.

\vspace{2mm}

\begin{lemma}[{\cite[Lemma 2.3]{LX22}}]\label{LX22, 2.3}
     Let $X\rightarrow Z$ be a projective morphism from a normal variety to a variety, and $D$ an $\mathbb{R}$-Cartier divisor on $X$.
     Then the following statements hold.\vspace{-3mm}
     \begin{itemize}
          \item[(1)] $\kappa_\sigma(X/Z,D)=\kappa_\sigma(X/Z,D')$ for any $\mathbb{R}$-Cartier divisor $D'$ with $D\equiv_ZD'$.\vspace{-2mm}
          \item[(2)] Let $f:Y\rightarrow X$ be a surjective birational morphism.
                     Then $\kappa_\sigma(X/Z,D)=\kappa_\sigma(Y/Z,f^*D+E)$ for any $f$-exceptional $\mathbb{R}$-divisor $E\ge0$.\vspace{-2mm}
          \item[(3)] Let $g:W\rightarrow X$ be a surjective projective morphism from a normal variety $W$.
                     Then $\kappa_\sigma(X/Z,D)=\kappa_\sigma(W/Z,g^*D)$.\vspace{-2mm}
          \item[(4)] Let $\phi:X\dashrightarrow X'$ be a partial $D$-MMP/$Z$. Then $\kappa_\sigma(X/Z,D)=\kappa_\sigma(X'/Z,\phi_*D)$.
     \end{itemize}
\end{lemma}

\noindent\textbf{Nakayama-Zariski decomposition.} We adopt the definition of Nakayama-Zariski decomposition in \cite[Section 3]{LX22}, although in this paper
we deal with Nakayama-Zariski decompositions only in the absolute setting.
\par Let $X$ be a normal projective variety and $P$ a prime divisor \emph{on} $X$.
For any big $\mathbb{R}$-Cartier divisor $B$, we define
$$\sigma_P(B):=\textup{inf}\{\textup{mult}_PB'\ |\ 0\le B'\sim_\mathbb{R}B\}.$$
Let $D$ be a pseudo-effective $\mathbb{R}$-Cartier divisor on $X$ and $A$ an ample $\mathbb{R}$-divisor on $X$. We define
$$\sigma_P(D):=\textup{lim}_{\epsilon\rightarrow 0^+}\sigma_P(D+\epsilon A).$$
Note that $0\le\sigma_P(D)<+\infty$ and that it does not depend on the choice of $A$ (\cite[III, 1.5 Lemma]{Nak-book}). We define
$$N_\sigma(D):=\sum_{P:\textup{ prime divisor on $X$}}\sigma_P(D)P.$$
Then $N_\sigma(D)$ is an effective $\mathbb{R}$-divisor. Note that $N_\sigma(D)$ is not necessarily $\mathbb{R}$-Cartier. We then define the $\mathbb{R}$-divisor
$$P_\sigma(D):=D-N_\sigma(D).$$
The decomposition $D=P_\sigma(D)+N_\sigma(D)$ is called the \emph{Nakayama-Zariski decomposition} of $D$.
\par When $X$ is smooth, the definition of $N_\sigma(D)$ coincides with \cite[III, 1.12 Definition]{Nak-book}.
Moreover, \cite[Lemma 3.4(3)]{LX22} shows that this definition coincides with the one defined in \cite[\S 4]{BH14}.

\begin{lemma}\label{00}
     Let $X$ be a normal projective variety and $D$ a pseudo-effective
     $\mathbb{R}$-Cartier divisor on $X$. Assume that $N_\sigma(D)$ is $\mathbb{R}$-Cartier and $\kappa_\sigma(X,D)=0$.
     Then $D\equiv N_\sigma(D)$.
     \par In particular, if $D$ is movable (i.e., $N_\sigma(D)=0$), then $\kappa_\sigma(X,D) = 0$ if and only if $D\equiv0$.
          If additionally $D\ge0$, then $\kappa_\sigma(X,D)=0$ if and only if $D=0$.
\end{lemma}
     \begin{proof}
          Let $f:Y\rightarrow X$ be a resolution. By Lemma \ref{LX22, 2.3} we have $\kappa_\sigma(Y,f^*D)=\kappa_\sigma(X,D)=0$,
          and so $f^*D\equiv N_\sigma(f^*D)$ by \cite[V, 1.12. Corollary]{Nak-book} .
          Since $N_\sigma(D)$ is $\mathbb{R}$-Cartier by assumption, it follows by \cite[Lemma 3.4(3)]{LX22} that
          $$D=f_*f^*D\equiv f_*N_\sigma(f^*D)=N_\sigma(D).$$
          This proves the first statement.
          \par Assume that $D$ is movable. Then $N_\sigma(D)=0$ is $\mathbb{R}$-Cartier.
          Hence, by the above argument $\kappa_\sigma(X,D) = 0$ implies that $D\equiv N_\sigma(D)=0$.
          Conversely, if $D\equiv0$, then Lemma \ref{LX22, 2.3}(1) and the definition of $\kappa_\sigma$ immediately imply that $\kappa_\sigma(X,D)=0$.
          If additionally $D\ge0$, then $D\equiv0$ if and only if $D=0$. Thus the second statement follows.
     \end{proof}

\noindent\textbf{b-divisors.}
Let $X$ be a normal quasi-projective variety. We call $Y$ a \emph{birational model} over $X$ if there exists a birational projective morphism $Y\rightarrow X$.
\par Let $X\dashrightarrow X'$ be a birational map. For any divisorial valuation $\nu$ over $X$, we define $\nu_{X'}$ to be the center of $\nu$ on $X'$.
An \emph{$\mathbb{R}$-b-divisor} $\mathbf{D}$ over $X$ is a formal sum $\mathbf{D} = \sum_\nu r_\nu\nu$ where $\nu$ are divisorial valuations over $X$ and $r_\nu\in\mathbb{R}$,
such that $\nu_X$ is not a divisor except for finitely many $\nu$. The \emph{trace} of $\mathbf{D}$ on $X'$ is the $\mathbb{R}$-divisor
$$\mathbf{D}_{X'}:=\sum_{\nu_{i,X'}\textup{ is a divisor}}r_i\nu_{i,X'}.$$
If $\mathbf{D}_{X'}$ is $\mathbb{R}$-Cartier and $\mathbf{D}_Y$ is the pullback of $\mathbf{D}_{X'}$ on $Y$ for any birational model $Y$ of $X'$,
we say that $\mathbf{D}$ \emph{descends} to $X'$, and that $\mathbf{D}$ is the \emph{closure} of $\mathbf{D}_{X'}$, and write $\mathbf{D}=\overline{\mathbf{D}_{X'}}$.
\par Let $\mathbf{D}$ be an $\mathbb{R}$-b-divisor over $X$ such that $\mathbf{D}$ descends to some birational model $Y$ over $X$.
If $\mathbf{D}_Y$ is nef, then we say that $\mathbf{D}$ is \emph{b-nef}.
If $\mathbf{D}_Y$ is an $\mathbb{R}$(resp. $\mathbb{Q}$)-Cartier divisor, then we say that $\mathbf{D}$ is \emph{$\mathbb{R}$\textup{(resp.} $\mathbb{Q}$\textup{)}-b-Cartier}.
If $\mathbf{D}_Y$ is a Cartier divisor, then we say that $\mathbf{D}$ is \emph{b-Cartier}.
\vspace{2mm}

\noindent\textbf{Generalized pairs.}
A \emph{generalized sub-pair} (\emph{{g-sub-pair}} for short) $(X,B,\mathbf{M})$ consists of a normal quasi-projective variety $X$,
an $\mathbb{R}$-divisor $B$ on $X$, and a b-nef $\mathbb{R}$-b-divisor $\mathbf{M}$ over $X$, such that
$K_X+B+\mathbf{M}_X$ is $\mathbb{R}$-Cartier.
\par We say $(X,B,\mathbf{M})$ is \emph{projective} if $X$ is projective.
\par A g-sub-pair $(X,B,\mathbf{M})$ is called a \emph{g-pair} if $B\ge0$.

\vspace{2mm}

\noindent\textbf{Singularities.}
Let $(X,B,\mathbf{M})$ be a g-(sub-)pair. For any prime divisor $E$ and $\mathbb{R}$-divisor $D$ on $X$, denote by $\textup{mult}_ED$ the \emph{multiplicity} of $E$ along $D$.
Let $h:W\rightarrow X$ be any log resolution of $(X,\textup{Supp}B)$ such that $\mathbf{M}$ descends to $W$, and let
$$K_W+B_W+\mathbf{M}_W:=h^*(K_X+B+\mathbf{M}_X).$$
The \emph{log discrepancy} of a prime divisor $D$ on $W$ with respect to $(X,B,\mathbf{M})$ is $1-\textup{mult}_DB_W$ and it is denoted by $a(D,X,B,\mathbf{M})$.
\par We say that $(X,B,\mathbf{M})$ is \emph{g-(sub-)lc} (resp. \emph{g-(sub-)klt}) if $a(D,X,B,\mathbf{M})\ge0$ (resp. $>0$) for any log resolution $h:W\rightarrow X$
as above and any prime divisor $D$ on $W$.
\par We say that $(X,B,\mathbf{M})$ is \emph{g-dlt} if there exists a log resolution $h:W\rightarrow X$ such that $a(D,X,B,\mathbf{M})>0$ for every exceptional/$X$ divisor on $W$.
Note that by \cite{Hu25} this definition of g-dlt pairs coincides with the one in \cite{Bir19}.
\vspace{2mm}

Recall that a contraction $f:X\rightarrow Y$ is called \emph{equidimensional} if all the fibers of $f$ have the same dimension.

\begin{theorem}[Equidimensional reduction{\cite[Theorem 2.1]{LX22}}]\label{equidimensional-reduction}
     Let $(X,B)$ be a dlt pair and $\pi:X\rightarrow Z$ be a projective surjective morphism over a normal variety $Z$.
     Then there exists a commutative diagram of projective morphisms
     $$
     \begin{tikzcd}
          Y \arrow[d,swap,"\pi_Y"] \arrow[r,"f"] & X \arrow[d,"\pi"]\\
          V \arrow[r,swap,"\varphi"] & Z
     \end{tikzcd}
     $$
     such that
     \begin{itemize}
          \item $f$, $\varphi$ are birational morphisms, $\pi_Y$ is an equidimensional contraction, $Y$ is $\mathbb{Q}$-factorial and klt,
          and $V$ is a smooth quasi-projective variety, and\vspace{-2mm}
          \item there exist two $\mathbb{R}$-divisors $B_Y$ and $E$ on $Y$ such that\vspace{-2mm}
          \begin{itemize}
               \item[(a)] $K_Y+B_Y=f^*(K_X+B)+E$,\vspace{-1mm}
               \item[(b)] $B_Y\ge0$, $E\ge0$ is exceptional/$X$, and $B_Y$ and $E$ have no common components,\vspace{-1mm}
               \item[(c)] $(Y,B_Y)$ is lc,
               and the image of any lc center of $(Y,B_Y)$ on $X$ is an lc center of $(X,B)$.\vspace{-2mm}
          \end{itemize} 
     \end{itemize}
\end{theorem}

\begin{lemma}\label{vertical to very exceptional}
     Let $\pi:X\rightarrow Z$ be an equidimensional contraction from a normal variety to a normal $\mathbb{Q}$-factorial variety.
     Let $D$ be a vertical over $Z$ $\mathbb{R}$-Cartier divisor on $X$.
     \par Then, there exists an effective and very exceptional/$Z$ $\mathbb{R}$-Cartier divisor $D'$ on $X$ such that $D\sim_{\mathbb{R},Z}D'$.
     If additionally $D$ is nef/$Z$, then $D\sim_{\mathbb{R},Z}0$.
\end{lemma}
     \begin{proof}
          Since $Z$ is $\mathbb{Q}$-factorial, for any prime divisor $P\subset Z$ we can define
          $$t_P:=\textup{inf}\{t\in\mathbb{R}\ |\ D+t\pi^*P\ge0\}.$$
          Note that $t_P\ge0$.
          Moreover, since $\pi$ is equidimensional and $D$ is vertical over $Z$, the image of every component of $D$ on $Z$ is a divisor,
          therefore $t_P>0$ for only finitely many prime divisors $P\subset Z$, and $\sum_i t_{P_i}P_i$
          is a well-defined $\mathbb{R}$-divisor where $P_i$ run over all prime divisors on $Z$.
          By construction, the divisor $D':=D+\pi^*(\sum_i t_{P_i}P_i)$ satisfies the requirements in the Lemma.
          Assume additionally $D$ is nef/$Z$. By applying \cite[Lemma 3.3]{Bir12} to $-D'$ we deduce that $-D'\ge0$, thus $D'=0$.
     \end{proof}

The following proposition can be viewed as a variation of \cite[Theorem 3.4]{Bir12} and \cite[Proposition 3.8]{HL18}.
\vspace{2mm}

\begin{proposition}\label{ve}
     Let $(X,B,\mathbf{M})$ be a g-lc pair such that $X$ is $\mathbb{Q}$-factorial and klt, and $\pi:X\rightarrow Z$ a contraction to normal quasi-projective variety.
     Assume that\vspace{-2mm}
     \begin{itemize}
          \item $\kappa_\sigma(X/Z,K_X+B+\mathbf{M}_X)=0$, and\vspace{-3mm}
          \item $K_X+B+\mathbf{M}_X\equiv_ZD^h+D^v$
                where $D^h\ge0$, $D^v\ge0$ and every component of $D^h$ (resp. $D^v$) is horizontal (resp. vertical) over $Z$.\vspace{-2mm}
     \end{itemize}
     Run a $(K_X+B+\mathbf{M}_X)$-MMP/$Z$ with scaling of an ample/$Z$ divisor
     $$X:=X_0\dashrightarrow X_1\dashrightarrow\cdots\dashrightarrow X_i\dashrightarrow\cdots$$
     and let $D^h_i$ (resp. $D^v_i$) be the birational transform of $D^h$ (resp. $D^v$) on $X_i$. Then, after finitely many steps,
     the MMP arrives at a model $X_n$ on which $K_{X_n}+B_{X_n}+\mathbf{M}_{X_n}\equiv_Z D^v_n$.
     \par Moreover, if additionally $D^v$ is very exceptional/$Z$, then the MMP terminates with a model $Y$ on which
     $K_Y+B_Y+\mathbf{M}_Y\equiv_Z0$. In particular, this holds when $\pi$ is an equidimensional contraction and $Z$ is $\mathbb{Q}$-factorial.
\end{proposition}
     \begin{proof}
          After finitely many steps, the MMP
          reaches a model $(X_n/Z,B_n,\mathbf{M})$ on which $N_\sigma((K_{X_n}+B_n+\mathbf{M}_{X_n})|_{F_n})=0$
          where $F_n$ is a very general fiber of the induced morphism $X_n\rightarrow Z$.
          Note that every component of $D^h_i$ (resp. $D^v_i$) remains horizontal (resp. vertical) over $Z$ for every $i\ge0$.
          We have
          $$D_n^h|_{F_n}=(D_n^h+D_n^v)|_{F_n}\equiv(K_{X_n}+B_n+\mathbf{M}_{X_n})|_{F_n}$$
          and hence
          $$N_\sigma(D_n^h|_{F_n})=N_\sigma((K_{X_n}+B_n+\mathbf{M}_{X_n})|_{F_n})=0.$$
          Additionally, by Lemma \ref{LX22, 2.3} we have
          \begin{align*}
               \kappa_\sigma(F_n,D_n^h|_{F_n})=\kappa_\sigma(F_n,(K_{X_n}+B_n+\mathbf{M}_{X_n})|_{F_n})&=\kappa_\sigma(X_n/Z,K_{X_n}+B_n+\mathbf{M}_{X_n})\\
               &=\kappa_\sigma(X/Z,K_X+B+\mathbf{M}_X)\\
               &=0.
          \end{align*}
          Thus we get $D_n^h|_{F_n}=0$ by Lemma \ref{00}, which in turn implies that $D_n^h=0$ since $D_n^h\ge0$ is horizontal over $Z$.
          Therefore we have $K_{X_n}+B_n+\mathbf{M}_{X_n}\equiv_ZD^v_n$ and the first assertion is clear.
          \par Assume moreover that $D^v$ is very exceptional/$Z$. Then,
          since the MMP only contracts the components of $D^h$ and of $D^v$ and since $D^h$ is horizontal over $Z$,
          it follows that $D^v$ remains very exceptional/$Z$ during the MMP.
          Therefore, by \cite[Proposition 3.8]{HL18},
          the MMP terminates with a model $(Y,B_Y,\mathbf{M})$ over $Z$ on which $K_Y+B_Y+\mathbf{M}_Y\equiv_Z0$.
          In particular, if $D^v$ is vertical over $Z$ and if $\pi$ is an equidimensional contraction and $Z$ is $\mathbb{Q}$-factorial,
          then by Lemma \ref{vertical to very exceptional} we can find a very exceptional/$Z$ $\mathbb{R}$-Cartier divisor $D^e\ge0$
          such that $D^v\sim_{\mathbb{R},Z}D^e$.
     \end{proof}

We close this section by recalling the $\mathbb{R}$-version of Filipazzi's result.

\begin{theorem}[{\cite[Theorem 2.23]{JLX22}}]\label{g-cbf}
     Let $(X,B,\mathbf{M})$ be a g-sub-pair and let $f:X\rightarrow Z$ be a contraction to a normal projective variety $Z$ such that\vspace{-2mm}
     \begin{itemize}
          \item $(X,B,\mathbf{M})$ is g-lc over the generic point of $Z$, and\vspace{-3mm}
          \item $K_X+B+\mathbf{M}_X\sim_{\mathbb{R},Z}0$.\vspace{-2mm}
     \end{itemize}
     Then, there is a g-sub-pair $(Z,B_Z,\mathbf{N})$ such that
     $$K_X+B+\mathbf{M}_X\sim_\mathbb{R}f^*(K_Z+B_Z+\mathbf{N}_Z).$$
     Moreover, if $(X,B,\mathbf{M})$ is a g-pair (resp. g-lc pair, g-klt pair), then so is $(Z,B_Z,\mathbf{N})$.
\end{theorem}

\textsc{\section{A numerical generalized canonical bundle formula}\label{Sec: A numerical generalized canonical bundle formula}}

This section aims to prove Theorem \ref{n-g-cbf}, starting with the special case where $\mathbf{M}=0$.
As mentioned in the introduction, the $\mathbb{Q}$-factoriality assumption is not required here.\vspace{2mm}

\begin{theorem}[Numerical canonical bundle formula]\label{cbf}
     Let $(X,B)$ be a sub-pair and let $f:X\rightarrow Z$ be a contraction to a normal variety $Z$ such that\vspace{-2mm}
     \begin{itemize}
          \item $(X,B)$ is lc over the generic point of $Z$, and\vspace{-3mm}
          \item $K_X + B\equiv_Z 0$.\vspace{-2mm}
     \end{itemize}
     Then, there exists a g-sub-pair $(Z,B_Z,\mathbf{N})$ such that
     $$K_X+B\sim_\mathbb{R}f^*(K_Z+B_Z+\mathbf{N}_Z).$$
     Moreover, if $(X,B)$ is a pair (resp. lc pair, klt pair), then $(Z,B_Z,\mathbf{N})$ is a g-pair (resp. g-lc pair, g-klt pair).
\end{theorem}

The following lemma is useful:

\begin{lemma}\label{HX16, 1.6}
     Let $(X,B)$ be an lc pair and $X\to Z$ a projective morphism such that $K_X+B\equiv_Z0$. Then $K_X+B\sim_{\mathbb{R},Z}0$.
\end{lemma}
     \begin{proof}
          If $B$ is a $\mathbb{Q}$-divisor, then the statement is nothing but a special case of \cite[Corollary 1.6]{HX16}.
          In general, by replacing $(X,B)$ with a dlt blow-up we may assume that $(X,B)$ is $\mathbb{Q}$-factorial dlt.
          Moreover, since $K_X+B$ is nef/$Z$, by using Shokurov's polytopes
          (See \cite[Remark 3.1 and Proposition 3.2]{Bir11}) we have a decomposition
          $$K_X+B=\sum_{i=1}^kr_i(K_X+B_i)$$
          such that for every $i$,\vspace{-2mm}
          \begin{itemize}
               \item $B_i\ge0$ is a $\mathbb{Q}$-divisor,\vspace{-3mm}
               \item $(X,B_i)$ is lc,\vspace{-3mm}
               \item $K_X+B_i$ is nef/$Z$, and\vspace{-3mm}
               \item $r_i\in\mathbb{R}$, $r_i\in(0,1)$, and $\sum_{i=1}^kr_i=1$.\vspace{-2mm}
          \end{itemize}
          Since $K_X+B\equiv_Z0$ and since $r_i>0$ and $K_X+B_i$ is nef/$Z$ for every $i$, we deduce that $K_X+B_i\equiv_Z0$ for every $i$.
          Therefore, by \cite[Corollary 1.6]{HX16} again we have $K_X+B_i\sim_{\mathbb{Q},Z}0$ for every $i$,
          and so $K_X+B\sim_{\mathbb{R},Z}0$.
     \end{proof}

     \begin{proof}[Proof of Theorem \ref{cbf}]
          Since $(X,B)$ is lc over the generic point of $Z$,
          possibly by replacing $B$ with $B+f^*H-f^*H'$ for some general ample divisors $H$ and $H'$ on $Z$, we can assume that $(X,B)$ is lc.
          By Lemma \ref{HX16, 1.6} we have $K_X + B\sim_{\mathbb{R},Z} 0$, hence the statement follows from Theorem \ref{g-cbf}.
     \end{proof}

We now turn to the general case and begin with some definitions.\vspace{2mm}

\noindent\textbf{Discrepancy b-divisor.} Let $(X,B,\mathbf{M})$ be a g-sub-pair. We define the b-divisors $\mathbf{A}(X,B,\mathbf{M})$
and $\mathbf{A}^*(X,B,\mathbf{M})$ in the following way: for any birational morphism $f : Y\rightarrow X$ from a normal variety $Y$, we define
$$\mathbf{A}(X,B,\mathbf{M})_Y:=K_Y+\mathbf{M}_Y-f^*(K_X+B+\mathbf{M}_X),\textup{ and }$$
$$\mathbf{A}^*(X,B,\mathbf{M})_Y:=\mathbf{A}(X,B,\mathbf{M})_Y^{>-1}.$$
The b-divisor $\mathbf{A}(X,B,\mathbf{M})$ is called the \emph{discrepancy b-divisor} of $(X,B,\mathbf{M})$.\vspace{2mm}

\begin{definition}[Numerical g-(sub-)lc trivial fibrations]
     \textup{
     A \emph{numerical g-sub-lc trivial fibration} $(f:X\rightarrow Z,B,\mathbf{M})$ consists of\vspace{-2mm}
     \begin{itemize}
          \item[(1)] a g-sub-pair $(X,B,\mathbf{M})$ and a contraction $f:X\rightarrow Z$ to a normal $\mathbb{Q}$-factorial variety $Z$;\vspace{-3mm}
          \item[(2)] $(X,B,\mathbf{M})$ is g-sub-lc over the generic point of $Z$;\vspace{-3mm}
          \item[(3)] $\textup{rank}f_*\mathcal{O}_X(\left\lceil \mathbf{A}^*(X,B,\mathbf{M})\right\rceil)=1$;\vspace{-3mm}
          \item[(4)] $K_X+B+\mathbf{M}_X\equiv_Z0$.\vspace{-2mm}
     \end{itemize}
     If additionally\vspace{-2mm}
     \begin{itemize}
          \item[(5)] $(X,B,\mathbf{M})$ is g-lc over the generic point of $Z$,\vspace{-2mm}
     \end{itemize}
     then we call it a \emph{numerical g-lc trivial fibration}.}
     \par\textup{We say $(f:X\rightarrow Z,B,\mathbf{M})$ is \emph{projective} if $X$ and $Z$ are projective.}
\end{definition}

     Note that if $B$ is effective over the generic point of $Z$, then $\mathcal{O}_X(\left\lceil \mathbf{A}^*(X,B,\mathbf{M})\right\rceil)=\mathcal{O}_X$.
     Hence for a numerical g-lc trivial fibration the condition (3) is automatically satisfied.
     
\begin{definition}[Discriminant and moduli parts. (cf.{\cite[3.4]{FG14}})]\label{dri.mod. part}
\textup{
     Let $(f:X\rightarrow Z,B,\mathbf{M})$ be a numerical g-sub-lc trivial fibration.
     In the following, we fix a choice of $K_X$ and a choice of $K_Z$, and assume that for any birational
     morphism $\phi: X' \rightarrow X$ and $\psi : Z'\rightarrow Z$, $K_{X'}$ and $K_{Z'}$ are chosen as the Weil divisors
     such that $\phi_*K_X' = K_X$ and $\psi_*K_{Z'} = K_Z$.
}
     \par\textup{Let $P$ be a prime divisor on $Z$. By shrinking $Z$ around the generic point of $P$ we may assume that $P$ is Cartier. We set
     $$b_P:=\textup{sup}\{t\in\mathbb{R}\ |\ (X,B+tf^*P,\mathbf{M})\textup{ is g-sub-lc over the generic point of $P$}\}$$
     and
     $$B_Z:=\sum_P (1-b_P)P$$
     where $P$ runs over all prime divisors on $Z$. Then, $B_Z$ is a well-defined $\mathbb{R}$-divisor on $Z$
     since $b_P = 1$ for all but a finite number of prime divisors (cf. \cite[Remark 3.1.3]{Amb99}).
     Assume that
     \begin{itemize}
          \item[($\ast$)] there exists an $\mathbb{R}$-Cartier divisor $L$ on $Z$ such that $K_X+B+\mathbf{M}_X\equiv f^*L$.
     \end{itemize}
     Then we set
     $$N_Z:=L-(K_Z+B_Z).$$
     Note that by the construction, $B_Z$ is uniquely determined by $(X,B,\mathbf{M})$, $N_Z$ is determined up to numerical equivalence,
     and $K_X+B+\mathbf{M}_X\equiv f^*(K_Z+B_Z+N_Z)$.
     We call $B_Z$ and $N_Z$ the \emph{pre-discriminant part} and the \emph{pre-moduli part} of $(f:X\rightarrow Z,B,\mathbf{M})$, respectively.
}
     \par \textup{Let $f':X'\rightarrow Z'$ be any contraction with birational projective morphisms $\phi: X' \rightarrow X$ and $\psi : Z'\rightarrow Z$
     (such constructions exist, for example, we can take $Z'\rightarrow Z$ as a resolution of $Z$ and then take
     $X'$ to be a resolution of the main component of $X\times_ZZ'$ which dominates $Z'$).
     Set $K_{X'}+B'+\mathbf{M}_{X'}:=\phi^*(K_X+B+\mathbf{M}_X)$. Then $(f':X'\rightarrow Z',B',\mathbf{M})$ is a numerical g-sub-lc trivial fibration
     induced by $(f:X\rightarrow Z,B,\mathbf{M})$ (cf.\cite[3.3]{FG14}). Set $L':=\psi^*L$, then we can define $B_{Z'}$ and $N_{Z'}$ as above.
     Since $\psi:Z'\rightarrow Z$ is an isomorphism over the codimension one points of $Z$, by Remark \ref{rem. crepant}(3) we have $\psi_*B_{Z'}=B_Z$, and hence $\psi_*N_{Z'}=N_Z$.
     In this way, we obtain the $\mathbb{R}$-b-divisors $\mathbf{B^Z}$ and $\mathbf{N}$ on $Z$, called the \emph{discriminant part}
     and the \emph{moduli part} of $(f:X\rightarrow Z,B,\mathbf{M})$, respectively.
     Note that the induced triplet $(Z,B_Z,\mathbf{N})$ is not yet a g-sub-pair as $\mathbf{N}$ is not necessarily b-nef.
}
\end{definition}

\begin{remark}\label{rem. crepant}
     \textup{Under the assumption and notation in Definition \ref{dri.mod. part}, we note the following facts:}
     \par\textup{(1) The definition of numerical g-sub-lc trivial fibrations is parallel to that of usual g-lc trivial fibrations,
     except that we further need $Z$ to be $\mathbb{Q}$-factorial.
     According to the constructions above, if a numerical g-sub-lc trivial fibration $(f:X\rightarrow Z,B,\mathbf{M})$ is also a g-lc trivial fibration
     (i.e. $K_X+B+\mathbf{M}_X\sim_{\mathbb{R},Z}0)$, and $B_Z'$, $N_Z'$ are its pre-discriminant part and pre-moduli part, respectively, then $B_Z=B_Z'$ and $N_Z\equiv N_Z'$.
     }
     \par\textup{(2) By construction, one can see that if $B$ is effective, then so is $B_Z$. In particular, suppose that $\mathbf{N}$ is b-nef,
     then $(Z,B_Z,\mathbf{N})$ is a g-pair if $(X,B,\mathbf{M}$) is a g-pair.
     Moreover, we claim that $(Z,B_Z,\mathbf{N})$ is g-lc (resp. g-klt, g-sub-lc, g-sub-klt)
     if $(X,B,\mathbf{M}$) is g-lc (resp. g-klt, g-sub-lc, g-sub-klt).
     We may assume that $\psi : Z'\rightarrow Z$ is a log resolution of $(Z,B_Z,\mathbf{N})$. Since $K_Z+B_Z+N_Z=L$
     and $K_{Z'}+B_{Z'}+N_{Z'}=\psi^*L$, it follows that $K_{Z'}+B_{Z'}+N_{Z'}=\psi^*(K_Z+B_Z+N_Z)$.
     As mentioned in Definition \ref{dri.mod. part}, we have
     $$B_{Z'}=\sum_{P'} (1-b_{P'})P'.$$
     where $P'$ runs over all prime divisors on $Z'$ and
     $$b_{P'}=\textup{sup}\{t\in\mathbb{R}\ |\ (X',B'+t(f')^*P',\mathbf{M})\textup{ is g-sub-lc over the generic point of $P'$}\}.$$
     If $(X,B,\mathbf{M})$ is g-lc (resp. g-klt, g-sub-lc, g-sub-klt),
     then $(X',B',\mathbf{M})$ is g-sub-lc (resp. g-sub-klt, g-sub-lc, g-sub-klt),
     which in turn implies $b_{P'}\ge0$ (resp. $b_{P'}>0$, $b_{P'}\ge0$, $b_{P'}>0$) for any prime divisor $P'$ on $Z'$.
     Therefore the coefficients of $B_{Z'}$ belong to $(-\infty,1]$ (resp. $(-\infty,1)$, $(-\infty,1]$, $(-\infty,1)$), and the claim follows.}
     \par\textup{(3) Let $(f:X\rightarrow Z,B,\mathbf{M})$ be a numerical g-sub-lc trivial fibration, and let $g:W\rightarrow X$
     be a birational model over $X$. Write $K_W+B_W+\mathbf{M}_W=g^*(K_X+B+\mathbf{M}_X)$.
     Then $(f\circ g:W\rightarrow Z,B_W,\mathbf{M})$ is a numerical g-sub-lc trivial fibration and $K_W+B_W+\mathbf{M}_W\equiv(f\circ g)^*L$.
     Since the definition of $B_Z$ is divisorial, the discriminant part and the moduli part of $(f\circ g:W\rightarrow Z,B_W,\mathbf{M})$
     are the same as the discriminant part and the moduli part of $(f:X\rightarrow Z,B,\mathbf{M})$. See also \cite[Remark 3.1.2]{Amb99}.}
     \par\textup{(4) As indicated in the last paragraph of Definition \ref{dri.mod. part}, when studying properties of $\mathbf{N}$ (such as b-nefness, b-semi-ampleness, etc.),
     we can freely replace $X\rightarrow Z$ with a fibration induced by generically finite base change. See also \cite[Remark 4.3]{Fil20}.} 
     \par\textup{(5) If $(f:X\rightarrow Z,B,\mathbf{M})$ is a numerical g-(sub-)lc trivial fibration and $G$ an $\mathbb{R}$-divisor on $Z$,
     then $(f:X\rightarrow Z,B+f^*G,\mathbf{M})$ is still a numerical g-(sub-)lc trivial fibration,
     with pre-discriminant part $B_Z+G$ and pre-moduli part $N_Z$ (by replacing $L$ with $L+G$).}
\end{remark}

\begin{lemma}\label{bir.tri.pullback}
     Let $f:X\rightarrow Z$ be a projective morphism between normal projective varieties such that\vspace{-2mm}
     \begin{itemize}
          \item $f$ is birational, and\vspace{-3mm}
          \item $Z$ is $\mathbb{Q}$-factorial.\vspace{-2mm}
     \end{itemize} 
     Suppose that $D$ is an $\mathbb{R}$-Cartier divisor on $X$ such that $D\equiv_Z0$. Then $D=f^*L$ for some $\mathbb{R}$-Cartier divisor $L$ on $Z$.
\end{lemma}
     \begin{proof}
          Since $Z$ is $\mathbb{Q}$-factorial, we can set $L:=f_*D$ and $E:=f^*L-D$. Then $E$ is exceptional over $Z$,
          and by the negativity lemma we see that $E=0$.
     \end{proof}

\begin{proposition}\label{trivial moduli part on a equidimensional reduction}
     Let $(X,B,\mathbf{M})$ be a projective $\mathbb{Q}$-factorial g-lc pair,
     and let $f:X\rightarrow Z$ be a contraction morphism to a normal projective variety $Z$ such that\vspace{-2mm}
     \begin{itemize}
          \item $K_X+B+\mathbf{M}_X\equiv_Z0$, and\vspace{-3mm}
          \item $\mathbf{M}_X|_{X_\eta}\equiv0$ where $X_\eta$ is the generic fiber of $f$.\vspace{-2mm}
     \end{itemize}
     Then, there is a commutative diagram of projective morphisms
      $$
          \begin{tikzcd}
          X' \arrow[d,swap,"f'"] \arrow[r,"\phi"] & X \arrow[d,"f"]\\
          Z' \arrow[r,swap,"\psi"] & Z
          \end{tikzcd}
     $$
     \begin{itemize}
          \item $\phi$, $\psi$ are birational, $f'$ is an equidimensional contraction, $X'$ is $\mathbb{Q}$-factorial and klt,
           $\mathbf{M}$ descends to $X'$, $Z'$ is a smooth projective variety, and\vspace{-3mm}
          \item there exist two $\mathbb{R}$-divisors $B_{X'}$ and $E'$ on $X'$ such that\vspace{-2mm}
               \begin{itemize}
                    \item[-] $K_{X'}+B_{X'}+\mathbf{M}_{X'}=\phi^*(K_X+B+\mathbf{M}_X)+E'$,\vspace{-1mm}
                    \item[-] $B_{X'}\ge0$, $E'\ge0$ is exceptional over $X$, $B_{X'}$ and $E'$ have no common components,\vspace{-1mm}
                    \item[-] $(X',B_{X'},\mathbf{M})$ is g-lc, and the image of any g-lc center of $(X',B_{X'},\mathbf{M})$ on $X$ is a g-lc center of $(X,B,\mathbf{M})$, and\vspace{-1mm}
                    \item[-] $\mathbf{M}_{X'}\equiv f'^*L'$ for some $\mathbb{R}$-Cartier divisor $L'$ on $Z'$, $K_{X'}+B_{X'}\equiv_{Z'}E'$, and $\kappa_\sigma(X'/Z',K_{X'}+B_{X'})=0$.\vspace{-1mm}
               \end{itemize} 
     \end{itemize}
\end{proposition}
     \begin{proof}
          By Theorem \ref{equidimensional-reduction}, there is a commutative diagram of projective morphisms
          $$
          \begin{tikzcd}
          X' \arrow[d,swap,"f'"] \arrow[r,"\phi"] & X \arrow[d,"f"]\\
          Z' \arrow[r,swap,"\psi"] & Z
          \end{tikzcd}
          $$
          such that\vspace{-2mm}
          \begin{itemize}
               \item $\phi$, $\psi$ are birational, $\phi$ factors through a sufficiently high log resolution of $(X,B,\mathbf{M})$ on which $\mathbf{M}$ descends,
               $f'$ is an equidimensional contraction, $X'$ is $\mathbb{Q}$-factorial and klt, $Z'$ is a smooth projective variety, and\vspace{-3mm}
               \item there exist two $\mathbb{R}$-divisors $B_{X'}$ and $E'$ on $X'$ such that\vspace{-2mm}
               \begin{itemize}
                    \item[-] $K_{X'}+B_{X'}+\mathbf{M}_{X'}=\phi^*(K_X+B+\mathbf{M}_X)+E'$,\vspace{-1mm}
                    \item[-] $B_{X'}\ge0$, $E'\ge0$ is exceptional over $X$, $B_{X'}$ and $E'$ have no common components, and\vspace{-1mm}
                    \item[-] $(X',B_{X'},\mathbf{M})$ is g-lc, and the image of any g-lc center of $(X',B_{X'},\mathbf{M})$ on $X$ is a g-lc center of $(X,B,\mathbf{M})$.\vspace{-2mm}
               \end{itemize} 
          \end{itemize}
          By the negativity lemma, we can write $\phi^*\mathbf{M}_X=\mathbf{M}_{X'}+F'$ where $F'\ge0$ is exceptional over $X$.
          \par We claim that $F'$ is vertical over $Z$. Denote by $X'_\eta$ the generic fiber of $f\circ\phi:X'\rightarrow Z$.
          We can assume $\textup{dim}X>\textup{dim}Z$ so that $\textup{dim}X'_\eta>0$.
          Assume to the contrary that $F'$ is horizontal over $Z$.
          Then we have $\textup{Supp}F'\bigcap X'_\eta\neq\emptyset$.
          Moreover, since $F'$ is exceptional over $X$, $F'|_{X'_\eta}$ is a non-zero effective divisor on $X'_\eta$.
          Hence, there is a curve $C'$ on $X'_\eta$ such that $F'\cdot C'>0$.
          On the other hand, since $\mathbf{M}_X|_{X_\eta}\equiv0$, we have
          $$0=\phi^*\mathbf{M}_X\cdot C'=\mathbf{M}_{X'}\cdot C'+F'\cdot C',$$
          which is impossible because $\mathbf{M}_{X'}$ is nef and $F'\cdot C'>0$. The claim follows.
          \par Next we show that $\mathbf{M}_{X'}\equiv f'^*L'$ for some $\mathbb{R}$-Cartier divisor $L'$ on $Z'$.
          Since $0\equiv\mathbf{M}_{X}|_{X_\eta}\equiv-(K_X+B)|_{X_\eta}$,
          by \cite{Gon11}, after passing to the algebraic closure of $k(Z)$ we have $-(K_X+B)|_{X_\eta}\sim_{\mathbb{R}}0$.
          So there is a vertical over $Z$ $\mathbb{R}$-divisor $G$ on $X$ such that
          $$\mathbf{M}_{X}\equiv_Z-(K_X+B)\sim_{\mathbb{R}}G.$$
          Therefore we have $\mathbf{M}_{X'}\equiv_Z G'$ where $G'=\phi^*G-F'$ is a vertical over $Z$ $\mathbb{R}$-divisor on $X'$.
          Since $f'$ is equidimensional and $Z'$ is smooth, by Lemma \ref{vertical to very exceptional}
          we have $\mathbf{M}_{X'}\equiv_{Z} G'\sim_{\mathbb{R},Z'}0$. The claim then follows by \cite[Theorem 1.2]{Leh15}.
          \par Now, we have
          $$K_{X'}+B_{X'}\equiv_{Z'}K_{X'}+B_{X'}+\mathbf{M}_{X'}=\phi^*(K_X+B+\mathbf{M}_X)+E'.$$
          Since $K_X+B+\mathbf{M}_X\equiv_Z0$, we see that $K_{X'}+B_{X'}\equiv_{Z'}E'$.
          Moreover, we have
          \begin{align*}
               \kappa_\sigma(X'/Z',K_{X'}+B_{X'})&=\kappa_\sigma(X'/Z',K_{X'}+B_{X'}+\mathbf{M}_{X'})\\
               &=\kappa_\sigma(X'/Z,K_{X'}+B_{X'}+\mathbf{M}_{X'})\\
               &=\kappa_\sigma(X/Z,K_X+B+\mathbf{M}_X)=0
          \end{align*}
          where the first and the third equalities follow by Lemma \ref{LX22, 2.3}(1) and (2), respectively; and the second equality is due to the birationality of $\psi$.
     \end{proof}

\begin{lemma}\label{Fil20.5.1}
     Let $(f:X\rightarrow Z,B,\mathbf{M})$ be a numerical g-sub-lc trivial fibration satisfying assumption $(\ast)$ in Definition \ref{dri.mod. part},
     and let $g:X\rightarrow Y$ be a contraction over $Z$.
     Suppose that $(g:X\rightarrow Y,B,\mathbf{M})$ is a numerical g-sub-lc trivial fibration such that
     $(h:Y\rightarrow Z,\mathbf{B}^{\mathbf{Y}}_Y,\mathbf{M^Y})$ is a numerical g-sub-lc trivial fibration, where $\mathbf{B^Y}$ and $\mathbf{M^Y}$
     are the discriminant part and the moduli part of $(g:X\rightarrow Y,B,\mathbf{M})$, respectively. Let
     $\mathbf{M^Z}$ be the moduli part of $(f:X\rightarrow Z,B,\mathbf{M})$ and $\mathbf{N}^\mathbf{Z}$ the moduli part of
     $(h:Y\rightarrow Z,\mathbf{B}^{\mathbf{Y}}_Y,\mathbf{M}^\mathbf{Y})$. Then $\mathbf{M^Z}=\mathbf{N^Z}$.
\end{lemma}
     \begin{proof}
          The proof of \cite[Lemma 5.1]{Fil20} works verbatim.
          We note that, if $L$ is the $\mathbb R$-Cartier divisor on $Z$ such that
          $$K_X+B+\mathbf M_X \equiv f^*L,$$
          then by assumption $K_X+B+\mathbf{M}_X\equiv g^*(h^*L)$ and $K_Y+\mathbf{B}^{\mathbf{Y}}_Y+\mathbf{M}^\mathbf{Y}_Y\equiv h^*L$,
          hence the discriminant and moduli parts of $(g:X\rightarrow Y,B,\mathbf{M})$ and $(h:Y\rightarrow Z,\mathbf{B}^{\mathbf{Y}}_Y,\mathbf{M^Y})$ are well defined.
     \end{proof}

\begin{theorem}\label{b-nef}
     Let $(f:X\rightarrow Z,B,\mathbf{M})$ be a projective numerical g-lc trivial fibration.
     Then, there exists an $\mathbb{R}$-Cartier divisor $L$ on $Z$ such that $K_X+B+\mathbf{M}_X\equiv f^*L$.
     Moreover, the moduli part $\mathbf{N}$ is b-nef.
\end{theorem}
     \begin{proof}
          Since $(X,B,\mathbf{M})$ is g-lc over the generic point of $Z$, we can choose some general ample divisors $H$ and $H'$ on $Z$ such that $(X,B+f^*H-f^*H',\mathbf{M})$ is g-lc.
          By Remark \ref{rem. crepant}(5), we can replace $(X,B,\mathbf{M})$ with $(X,B+f^*H-f^*H',\mathbf{M})$ to assume that $(X,B,\mathbf{M})$ is g-lc.
          Furthermore, by Remark \ref{rem. crepant}(3), replacing $(X,B,\mathbf{M})$ with its g-dlt blow-up we can assume that $(X,B,\mathbf{M})$ is $\mathbb{Q}$-factorial g-dlt.
          In particular, $(X,B)$ is $\mathbb{Q}$-factorial dlt.
          Let $X_\eta$ be the generic fiber of $f:X\rightarrow Z$. We distinguish two cases.\vspace{2mm}\\
          \noindent\textbf{Case I:} $\mathbf{M}_X|_{X_\eta}\equiv 0$.
          \par By Proposition \ref{trivial moduli part on a equidimensional reduction} there is a commutative diagram of projective morphisms
          $$
          \begin{tikzcd}
               X' \arrow[d,swap,"f'"] \arrow[r,"\phi"] & X \arrow[d,"f"]\\
               Z' \arrow[r,swap,"\psi"] & Z
          \end{tikzcd}
          $$
          such that
          \begin{itemize}
               \item $\phi$, $\psi$ are birational, $f'$ is an equidimensional contraction, $X'$ is $\mathbb{Q}$-factorial and klt, $Z'$ is a smooth projective variety,\vspace{-3mm}
               \item $\mathbf{M}$ descends to $X'$, and\vspace{-3mm}
               \item there exist two $\mathbb{R}$-divisors $B_{X'}$ and $E'$ on $X'$ such that\vspace{-2mm}
               \begin{itemize}
               \item[-] $K_{X'}+B_{X'}+\mathbf{M}_{X'}=\phi^*(K_X+B+\mathbf{M}_X)+E'$,\vspace{-1mm}
               \item[-] $B_{X'}\ge0$, $E'\ge0$ is exceptional/$X$, and $B_{X'}$ and $E'$ have no common components,\vspace{-1mm}
               \item[-] $(X',B_{X'},\mathbf{M})$ is g-lc, and the image of any g-lc center of $(X',B_{X'},\mathbf{M})$ on $X$ is a g-lc center of $(X,B,\mathbf{M})$, and\vspace{-1mm}
               \item[-] $\mathbf{M}_{X'}\equiv f'^*L'$ for some $\mathbb{R}$-Cartier divisor $L'$ on $Z'$, $K_{X'}+B_{X'}\equiv_{Z'}E'$, and $\kappa_\sigma(X'/Z',K_{X'}+B_{X'})=0$.\vspace{-1mm}
               \end{itemize}
          \end{itemize}
          By applying Proposition \ref{ve} to $f':X'\rightarrow Z'$, we can run a $(K_{X'}+B_{X'})$-MMP/$Z'$
          with scaling of an ample/$Z'$ divisor which contracts $E'$ and terminates with a model
          $\phi':(X',B_{X'})\dashrightarrow (Y',B_{Y'})$ over $Z'$
          with the induced morphism $f_{Y'}:Y'\rightarrow Z'$ on which $K_{Y'}+B_{Y'}\equiv_{Z'}0$.
          $$
          \begin{tikzcd}
          Y'\arrow[dr,swap,"f_{Y'}"] & \arrow[l,swap,dashrightarrow, "\phi'"] X' \arrow[d,swap,"f'"] \arrow[r,"\phi"] & X \arrow[d,"f"]\\
          &Z' \arrow[r,swap,"\psi"] & Z \\
          \end{tikzcd}
          $$
          Theorem \ref{cbf} yields a g-lc pair $(Z',B_{Z'},\mathbf{N})$ such that
          $$K_{Y'}+B_{Y'}\equiv f_{Y'}^*(K_{Z'}+B_{Z'}+\mathbf{N}_{Z'}).$$
          In particular, $\mathbf{N}$ is b-nef.
          We note that $\mathbf{N}$ is induced by the usual g-lc trivial fibration $(Y',B_{Y'})\rightarrow Z'$,
          which, as mentioned in Remark \ref{rem. crepant}(1), is exactly the moduli part of the numerical g-lc trivial fibration $(f_{Y'}:Y'\rightarrow Z',B_{Y'})$.
          Hence, 
          $$K_{Y'}+B_{Y'}+\mathbf{M}_{Y'}\equiv f_{Y'}^*(K_{Z'}+B_{Z'}+L'+\mathbf{N}_{Z'})$$
          where $\mathbf{N}$ is also the moduli part of $(f_{Y'}:Y'\rightarrow Z,B_{Y'},\mathbf{M})$, as mentioned in Remark \ref{rem. crepant}(5).
          \par Let $p:W\rightarrow X$, $q:W\rightarrow X'$, $r:W\rightarrow Y'$ be a common resolution of $X$, $X'$ and $Y'$.
          $$
          \begin{tikzcd}
          &  W \arrow[dl,swap,"r"] \arrow[d,swap,"q"] \arrow[dr,"p"] &\\
          Y'\arrow[dr,swap,"f_{Y'}"] & \arrow[l,swap,dashrightarrow, "\phi'"] X' \arrow[d,swap,"f'"] \arrow[r,"\phi"] & X \arrow[d,"f"]\\
          &Z' \arrow[r,swap,"\psi"] & Z\\
          \end{tikzcd}
          $$
          As $K_{X'}+B_{X'}+\mathbf{M}_{X'}-E'\equiv_Z0$, the MMP is $(K_{X'}+B_{X'}+\mathbf{M}_{X'}-E')$-trivial/$Z'$. Hence, by the negativity
          lemma we have
          $$r^*(K_{Y'}+B_{Y'}+\mathbf{M}_{Y'})=q^*(K_{X'}+B_{X'}+\mathbf{M}_{X'}-E').$$
          Consequently,
          $$K_{Z'}+B_{Z'}+L'+\mathbf{N}_{Z'}\equiv_Z 0.$$
          The first assertion follows from Lemma \ref{bir.tri.pullback}, and the second follows from Remark \ref{rem. crepant}(3)(4).
          \vspace{2mm}\\
          \noindent\textbf{Case II:} $\mathbf{M}_X|_{X_\eta}\not\equiv 0$.
          \par  We will prove the theorem by induction on the relative dimension $d:=\textup{dim}X-\textup{dim}Z$.
          \par In the case where $d=0$, $f$ is birational as it is a contraction, and so the first assertion follows immediately by Lemma \ref{bir.tri.pullback}.
          Furthermore, we have $f_*\mathbf{M}_X=\mathbf{N}_Z$, and so $\mathbf{N}=\mathbf{M}$ is b-nef. From now on we assume that $d>0$.
          \par Since $\mathbf{M}_X$ is pseudo-effective (as it is the pushdown of a nef divisor on a higher birational model),
          $K_X+B\equiv_Z -\mathbf{M}_X$ is not pseudo-effective/$Z$, and so we can run a $(K_X+B)$-MMP/$Z$ with scaling of an ample/$Z$ divisor which terminates with a Mori fiber space $g:X''\rightarrow Y$ over $Z$
          such that $K_{X''}+B''+\mathbf{M}_{X''}\equiv_Y 0$. Moreover, the contraction theorem yields $K_{X''}+B''+\mathbf{M}_{X''}\sim_{\mathbb{R},Y}0$.
          Therefore, by applying Theorem \ref{g-cbf} to $g:X''\rightarrow Y$, we induce a g-lc pair $(Y,B_Y,\mathbf{M}^\mathbf{Y})$
          such that $K_{X''}+B''+\mathbf{M}_{X''}\sim_\mathbb{R}g^*(K_Y+B_Y+\mathbf{M}^\mathbf{Y}_Y)$.
          \par We note that by construction $(X'',B'',\mathbf{M})$ is g-lc over the generic point of $Y$ and $Y$ is $\mathbb{Q}$-factorial,
          thus we naturally have a numerical g-lc trivial fibration $(g:X''\rightarrow Y,B'',\mathbf{M})$.
          Denote by $h:Y\rightarrow Z$ the induced contraction.
          By construction $(Y,B_Y,\mathbf{M}^\mathbf{Y})$ is g-lc over the generic point of $Z$,
          and since $K_{X''}+B''+\mathbf{M}_{X''}\equiv_Z0$ and $K_{X''}+B''+\mathbf{M}_{X''}\sim_\mathbb{R}g^*(K_Y+B_Y+\mathbf{M}^\mathbf{Y}_Y)$, we have $K_Y+B_Y+\mathbf{M}^\mathbf{Y}_Y\equiv_Z0$.
          Therefore, we have a numerical g-lc trivial fibration $(h:Y\rightarrow Z,B_Y,\mathbf{M}^\mathbf{Y})$.
          \par In the case where $\textup{dim}Y=\textup{dim}Z$, $h$ is birational. Hence, by Lemma \ref{bir.tri.pullback} there is an $\mathbb{R}$-Cartier divisor $L$ on $Z$ such that
          $$K_Y+B_Y+\mathbf{M}^\mathbf{Y}_Y=h^*L,$$
          and so
          $$K_{X''}+B''+\mathbf{M}_{X''}\sim_\mathbb{R}g^*h^*L.$$
          Consequently, it makes sense to consider the discriminant and moduli parts of $(X''\rightarrow Z,B'',\mathbf{M})$ and $(h:Y\rightarrow Z,B_Y,\mathbf{M}^\mathbf{Y})$.
          By the same reason as in the beginning of the proof,
          the moduli part of $(h:Y\rightarrow Z,B_Y,\mathbf{M}^\mathbf{Y})$ is b-nef.
          Moreover, by Lemma \ref{Fil20.5.1}, the moduli part of $(h:Y\rightarrow Z,B_Y,\mathbf{M}^\mathbf{Y})$ is exactly the moduli part of $(X''\rightarrow Z,B'',\mathbf{M})$.
          Take a common log resolution $p:W\rightarrow X$, $q:W\rightarrow X''$.
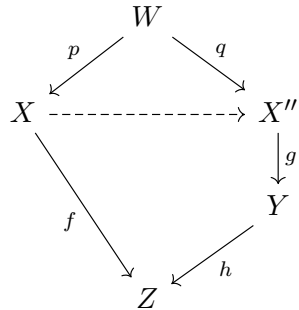
          $$
          \begin{tikzcd}
          &  W \arrow[dl,swap,"p"] \arrow[dr,"q"] &\\
          X \arrow[ddr,swap,"f"] \arrow[rr, dashrightarrow] & & X'' \arrow[d,"g"]\\
          & & Y \arrow[dl,"h"]\\
          & Z & 
          \end{tikzcd}
          $$
          Since the above $(K_X+B)$-MMP is $(K_X+B+\mathbf{M}_X)$-trivial over $Z$ (as $K_X+B+\mathbf{M}_X\equiv_Z0$) and $K_{X''}+B''+\mathbf{M}_{X''}$ is the birational transform of $K_X+B+\mathbf{M}_X$,
          by the negativity lemma we get
          $$p^*(K_X+B+\mathbf{M}_X)=q^*(K_{X''}+B''+\mathbf{M}_{X''}).$$
          Therefore by Remark \ref{rem. crepant}(3) we are done.
          In particular, as $\textup{dim}X>\textup{dim}Y$, this proves the case when $d=1$.
          \par Thus, we can assume that $\textup{dim}Y>\textup{dim}Z$.
          In this case, $(h:Y\rightarrow Z,B_Y,\mathbf{M}^\mathbf{Y})$ is a numerical g-lc trivial fibration with relative dimension less than $d$.
          Hence, by induction there is an $\mathbb{R}$-Cartier divisor $L$ on $Z$ such that
          $$K_Y+B_Y+\mathbf{M}^\mathbf{Y}_Y=h^*L,$$
          and the moduli part of $(h:Y\rightarrow Z,B_Y,\mathbf{M}^\mathbf{Y})$ is b-nef.
          By arguing as in the previous paragraph we prove the inductive step.
     \end{proof}

\begin{proof}[Proof of Theorem \ref{n-g-cbf}]
     Under the assumptions, we have a numerical g-lc trivial fibration $(f:X\rightarrow Z,B,\mathbf{M})$.
     The first assertion follows directly from Theorem \ref{b-nef} and the construction in Definition \ref{dri.mod. part}.
     The ``Moreover'' part follows from Remark \ref{rem. crepant}(2).
\end{proof}

\vspace{5mm}

\phantomsection     
\addcontentsline{toc}{section}{References}
\renewcommand\refname{\textsc{\large{References}}}

\noindent \textsc{Cheng Zhang}\\
\textsc{School of Mathematical Sciences}\\
\textsc{Xiamen University}\\
\textsc{Xiamen, 361001, China}\\
\itshape{Email address}: \textup{chengz23@stu.xmu.edu.cn}

\end{document}